\documentclass[pdflatex]{article}

\usepackage[utf8]{inputenc}
\usepackage{amsthm}
\usepackage{amsmath} 
\usepackage{amsxtra} 
\usepackage{mathrsfs}
\usepackage{amssymb} 
\usepackage{graphicx} 
\usepackage{stmaryrd} 
\usepackage{paralist} 
\usepackage{comment} 
\usepackage{mathtools}
\usepackage{ascmac}
\usepackage{tikz}
\usepackage{latexsym}
\usepackage{here}
\usepackage{bbm}

\usepackage[affil-it]{authblk}

\theoremstyle{definition} 

\newtheorem{theorem}{Theorem}[section]
\newtheorem{definition}[theorem]{Definition}
\newtheorem{lemma}[theorem]{Lemma}
\newtheorem{fact}[theorem]{Fact}
\newtheorem{prop}[theorem]{Proposition}

\newtheorem*{theorem*}{Theorem}
\newtheorem{claim}{Claim}[theorem]
\newtheorem*{claim*}{Claim}
\newtheorem*{subclaim*}{Subclaim}

\newtheorem{cor}[theorem]{Corollary}

\newenvironment{renumerate}%
{\begin{enumerate}}{\end{enumerate}}
{\begin{enumerate}}{\end{enumerate}}
\newenvironment{aenumerate}%
{\begin{enumerate}}{\end{enumerate}}

\newtheorem*{lemma*}{Lemma} 
\newtheorem*{prop*}{Proposition}

\newcommand{\mrm}{\mathrm} 
\newcommand{\mbb}{\mathbb} 

\newcommand{\dom}{\mathop{\mathrm{dom}} \nolimits}

\newcommand{\id}{\mathop{\mathrm{id}} \nolimits}

\newcommand{\height}{\mathop{\mathrm{height}} \nolimits}

\newcommand{\lchon}{\textrm{``}}
\newcommand{\rchon}{\textrm{''}}

\newcommand{\mbbB}{\mathbb{B}}
\newcommand{\mbbC}{\mathbb{C}}

\newcommand{\mbbP}{\mathbb{P}}
\newcommand{\mbbQ}{\mathbb{Q}} 
\newcommand{\mbbR}{\mathbb{R}} 
\newcommand{\mbbS}{\mathbb{S}}

\newcommand{\mcalH}{\mathcal{H}}

\newcommand{\mcalP}{\mathcal{P}}

\newcommand{\restrict}{\upharpoonright} 
\newcommand{\bool}[1]{\llbracket {#1} \rrbracket}

\newcommand{\coll}{\mathop{\mathrm{Coll}} \nolimits}
\newcommand{\add}{\mathop{\mathrm{Add}} \nolimits}

\newcommand{\ZFC}{\mathsf{ZFC}} 
\newcommand{\ZF}{\mathsf{ZF}} 

\newcommand{\AD}{\mathsf{AD}} 
\newcommand{\AC}{\mathsf{AC}} 
\newcommand{\DC}{\mathsf{DC}}

\newcommand{\Eval}{\mathrm{Eval}}

\newcommand{\Ord}{\mathrm{Ord}}

\newcommand{\LR}{L(\mathbb{R})}
\newcommand{\VRVG}{V( \mbbR^{V[G]})}
\newcommand{\VmuVG}{V(({\mu^\mu})^{V[G]})}
\newcommand{\VomegaVG}{V(({\omega^\omega})^{V[G]})}

\begin{document}

\title{Perfect set dichotomy theorem in  generalized Solovay model}

\author[1]{Hiroshi Sakai 
\thanks{This author is supported by JSPS Kakenhi Grant No. 24K06828.}
}
\author[2]{Toshimasa Tanno \thanks{This author is supported by JST SPRING Grant No. JPMJSP2148.}
}
\affil[1]{Graduate School of Mathematical Sciences, The University of Tokyo, 3-8-1 Komaba, Meguro-ku,  153-8914, Tokyo, Japan. \par
E-mail : hsakai@ms.u-tokyo.ac.jp}
\affil[2]{Graduate School of System Informatics, Kobe University, 1-1 Rokko-dai, Nada-ku, 657-8501, Kobe, Japan. \par
E-mail : 211x503x@stu.kobe-u.ac.jp}

\date{}

\maketitle

\begin{abstract}
  We prove that the perfect set dichotomy theorem holds in the Solovay model $\VomegaVG$. Namely, for every equivalence relation $E$ on $\mathbb{R}$, either $\mathbb{R}/E$ is well-orderable or there exists a perfect set consisting of $E$-inequivalent reals.   
  Furthermore we consider a  generalization of the Solovay model for an uncountable regular cardinal $\mu$
  and show the perfect set dichotomy theorem for $\mu^\mu$ also holds in that model.
  We establish the three element basis theorem for uncountable linear orders in the Solovay model for a weakly compact cardinal, in a general form  covering the uncountable case.
\end{abstract}

\section{Introduction}
 
Solovay \cite{solovay1970model} constructed a model of $\ZF + \DC$ in which all sets of reals have regularity properties such as the Baire property, Lebesgue measurability and the perfect set property. This model is called the Solovay model.
The Solovay model can be regarded as an ``ideal'' universe of mathematics in which there are no ``pathological'' sets of reals.
For example, T\"{o}rnquist \cite{tornquist2018definability} proved that there are no infinite mad families in the Solovay model.
In this model, every subset of the reals  has regularity properties of definable sets.

Among these properties, in this paper we focus on the \textit{perfect set dichotomy} for equivalence relations on the reals.
For an equivalence relation $E$ on the reals, the perfect set dichotomy asserts that either 
$\mathbb{R}/E$ is well-orderable or there exists a perfect set of pairwise $E$-inequivalent reals.
Classical results establish the dichotomy for definable equivalence relations.
Silver \cite{silver1980counting} proved that if $E$ is a $\mathbf{\Pi}^1_1$ equivalence relation then $\mbbR / E$ injects into $\omega$ or there is a perfect set of pairwise $E$-inequivalent reals.
Burgess \cite{burgess1979effective} showed that if $E$ is a $\mathbf{\Sigma}^1_1$ equivalence relation then $\mbbR / E$ injects into $\omega_1$ or there is a perfect set of pairwise $E$-inequivalent reals.
In the context of hypotheses that guarantee regularity properties for all sets of reals, 
Woodin showed that the perfect set dichotomy holds for all equivalence relations on $\mbbR$ under the assumption of $\ZF + \AD + V=L(\mathbb{R})$ (Theorem 3.2 in \cite{chan2021cardinality}).
In this paper we prove that the perfect set dichotomy holds for all equivalence relations on $\mathbb{R}$ in the Solovay model (Theorem \ref{dic_equi_Solovay}).

Furthermore, we consider a generalization of the Solovay model to higher cardinals.
Research on subsets of the generalized Baire space $\mu^\mu$ for an uncountable cardinal $\mu$ has been developed within the framework of generalized descriptive set theory.
The Solovay model is obtained by collapsing an inaccessible cardinal to $\omega_1$.
We consider the model obtained by collapsing an inaccessible cardinal $\kappa$ to $\mu^+$ for an uncountable regular cardinal $\mu$.
Schlicht \cite{MR3743612} showed the perfect set property for subsets of $\mu^\mu$ in this model.
We show that the perfect set dichotomy also holds for all equivalence relations on $\mu^\mu$ in this model (Theorem \ref{thm:psd_genSolovay}).

As an application of the perfect set dichotomy, we derive the basis theorem for uncountable linear orders in the Solovay model for a weakly compact cardinal.
Regarding linear orders in the Solovay model,
Chan and Jackson \cite{chan2019suslin} showed that if $\LR$ satisfies $\AD$ then in $\LR$ there are no Suslin lines.
That paper also presents Woodin's argument proving that there are no Suslin lines in the Solovay model for a weakly compact cardinal.
We show that there are no Aronszajn lines in such models.
More precisely, we prove that $\{(\omega_1, <), (\omega_1, >), (\mbbR, <_\mbbR)\}$ is a basis for uncountable linear orders in the Solovay model obtained by collapsing a weakly compact cardinal to $\omega_1$ (Corollary \ref{cor:three_element_basis_Sol}),
and in $\LR$ which satisfies $\AD$ (Theorem \ref{thm:three_element_basis_AD})\footnote{Weinert independently noticed that this fact holds in $\LR$ satisfying $\AD$.}.
We prove this theorem in a general form that also covers the generalized Solovay model (Theorem \ref{Solovay_basis}).

When we discuss the Solovay model, we work under the following settings, where $\coll(\mu, {<}\kappa)$ is the L\'{e}vy collapse forcing $\kappa$ to be $\mu^+$.
\begin{itemize}

\item[$(\ast)_{\mu,\kappa}$] $\mu$ is a regular cardinal and $\kappa$ is an inaccessible cardinal with $\mu < \kappa$ in $V$. 
$G$ is a $\coll(\mu, {<}\kappa)$-generic filter over $V$.

\item[$(\ast\ast)_{\mu, \kappa}$] 
$\mu$ is a regular cardinal and $\kappa$ is a weakly compact cardinal with $\mu < \kappa$ in $V$. 
$G$ is a $\coll(\mu, {<} \kappa)$-generic filter over $V$.
\end{itemize}

See \S \ref{sec:preliminaries} for more details about forcing notions and the Solovay model.
We will refer to the above settings in many places in this paper.

\section{Preliminaries} \label{sec:preliminaries}

Here we present our notation and basic facts used in this paper.

In this paper, we basically distinguish between $\omega^\omega$ and $\mathbb{R}$ 
(while implicitly using the existence of a definable bijection between them). 
$\omega^\omega$ is the set of all infinite countable sequences of elements of $\omega$, equipped with the topology generated by the basic open sets 
\begin{center}
  $N_s = \{ x \in \omega^\omega \mid s \subseteq x\}$
\end{center}
for $s \in \omega^{<\omega}$.
$\mathbb{R}$ is the real line $(\mbbR, <_{\mbbR})$ with the usual total order.

Next we give our notation on ordered sets.
We deal with several kinds of orderings.
A \textit{pre-ordering} is a binary relation which is reflexive and transitive.
A \textit{poset} is a pre-ordered set.
For a poset $\mbbP$, its domain is also denoted as $\mbbP$, and its order is denoted as $\leq_\mbbP$. The subscript $\mbbP$ in $\leq_\mbbP$ will be often omitted.
A \textit{partial ordering} is an antisymmetric pre-ordering.
A \textit{pre-linear ordering} is a pre-ordering by which any two elements are comparable.
A \textit{linear ordering} is an antisymmetric pre-linear ordering.

Let $\mbbP$ be a poset.
For $p,q \in \mbbP$ we write $p <_\mbbP q$ if $p \leq_\mbbP q$, and $q \not\leq_\mbbP p$.
For $p\in \mbbP$, $\mbb{P} \restrict p$ denotes the poset obtained by restricting $\mbb{P}$ to $\{ p' \in \mbb{P} \mid p' \leq_\mbbP p \}$.

Let $\mbbP$ and $\mbbQ$ be posets and $f$ be a function from $\mbbP$ to $\mbbQ$.
$f$ is said to be \textit{order preserving} if $f( p_0 ) \leq_\mbbQ f( p_1 )$ whenever $p_0 \leq_\mbbP p_1$. $f$ is said to be \textit{strictly order preserving} if $f( p_0 ) <_\mbbQ f( p_1 )$ whenever $p_0 <_\mbbP p_1$.
$f$ is called an \textit{embedding} if $f$ is an isomorphism from $P$ to $( f[P] , \leq_\mbbQ )$.
$\mbbP$ is 
\textit{embeddable} into $\mbbQ$ if there is  
an embedding from $\mbbP$ to $\mbbQ$.

Next we present our notation and basic facts on forcing.
In this paper, we are only concerned with forcings over trainsitive models of $\ZFC$.

Let $\mbb{P}$ and $\mbb{Q}$ be posets with the largest elements $1_\mbbP$ and $1_\mbbQ$, respectively.
A function $f : \mbb{Q} \to \mbb{P}$ is called a \emph{projection} if
\begin{renumerate}
\item $f$ is order preserving,
\item $f( 1_\mbb{Q} ) = 1_\mbb{P}$,
\item for any $q \in \mbb{Q}$ and any $p \in \mbb{P}$ with $p \leq_\mbbP f(q)$ there is $q' \leq_\mbbQ q$ such that $f( q' ) \leq_\mbbP p$.
\end{renumerate}

 A function $f : \mbb{Q} \to \mbb{P}$ is called a \emph{complete embedding} if
 \begin{renumerate}
\item $f$ is order preserving,
\item $f(1_\mbbQ) = 1_{\mbbP}$, 
\item $f( q_0 )$ and $f( q_1 )$ are incompatible in $\mbbP$
whenever $q_0$ and $q_1$ are incompatible in $\mbbQ$,
\item for any $p \in \mbbP$ there is $q \in \mbbQ$ such that 
for every $r \leq q$ in $\mbbQ$, $f(r)$ is compatible with $p$.
\end{renumerate}
$\mbbQ \subseteq \mbbP$ is \textit{complete subposet} if 
the inclusion map is a complete embedding.

A function $d : \mbb{Q} \to \mbb{P}$ is called a \emph{dense embedding} if
\begin{renumerate}
\item $d$ is order preserving,
\item $d( q_0 )$ and $d( q_1 )$ are incompatible in $\mbbP$
whenever $q_0$ and $q_1$ are incompatible in $\mbbQ$.
\item $d[ \mbb{Q} ]$ is dense in $\mbb{P}$, and $d( 1_\mbb{Q} ) = 1_\mbb{P}$.
\end{renumerate}
Notice that every dense embedding is a projection.

Suppose $d : \mbb{Q} \to \mbb{P}$ is a dense embedding in $V$. 
If $H$ is a $\mbb{Q}$-generic filter over $V$, then a filter $G \subseteq \mbb{P}$ generated by $d[H]$
is $\mbb{P}$-generic over $V$, and $V[G] = V[H]$.
If $G$ is a $\mbb{P}$-generic filter over $V$, then $H := d^{-1} [G]$ is a $\mbb{Q}$-generic filter over $V$,
and $V[H] = V[G]$.

Suppose $f \colon \mbb{P} \to \mbb{Q}$ is a complete embedding.
If $G$ is a $\mbbP$-generic filter over $V$,
then the quotient forcing $( \mbbQ / G )_f$ is defined as
\begin{center}
    $(\mbbQ / G)_f = 
    \{q\in \mbbQ \mid \forall p\in G \ f (p) \  || \  q\}$.
\end{center}

If $H$ is a $\mbbQ$-generic filter over $V$ then
$G := f^{-1}[H]$ is a $\mbbP$-generic filter over $V$.
If $H$ is a $\mbbQ / G$-generic filter over $V[G]$
then $H$ is a $\mbbP$-generic filter over $V$.

Suppose $f \colon \mbb{Q} \to \mbb{P}$ is a projection.
If $G$ is a $\mbbP$-generic filter over $V$,
then $( \mbbQ / G )^f$ is the poset in $V[G]$ obtained by restricting
$\mbbQ$ to $f^{-1} [G]$.
In $V$, $d \colon \mbb{Q} \to \mbb{P} * ( \mbb{Q} / \dot{G} )^f$ defined by $d(q) := ( f(q) , \check{q} )$ is a dense embedding, where
$\dot{G}$ is the canonical $\mbbP$-name for a $\mbbP$-generic filter.
If $H$ is a $\mbb{Q}$-generic filter over $V$,
then a filter $G \subseteq \mbb{P}$ generated by $f[H]$
is a $\mbb{P}$-generic filter over $V$, $H$ is a $( \mbb{Q} / G )^f$-generic filter over $V[G]$,
and $V[H] = V[G][H]$.

In this paper, a \emph{complete Boolean algebra} means a poset obtained by getting rid of the smallest element $0$ from a complete Boolean algebra in the usual sense.

We use the following standard fact. We give a proof since we could not find any reference.

\begin{lemma} \label{lem:projection}
Let $\mbb{P}$ be a complete Boolean algebra and $\mbb{Q}$ be a poset.
Suppose $H$ is a $\mbb{Q}$-generic filter over $V$, and in $V[H]$ there is a $\mbb{P}$-generic filter $G$ over $V$.
\begin{aenumerate}
\item For some $p^* \in G$ and $q^* \in H$, there is a projection $f \in V$
from $\mbb{Q} \restrict q^*$ to $\mbb{P} \restrict  p^*$ such that
$G \restrict p^*$ is a filter generated by $f[ H \restrict q^* ]$.
If $1_\mbbQ$ forces the existence of a $\mbbP$-generic filter
over $V$, then we can let $q^* = 1_\mbbQ$.
\item Suppose also that $H \in V[G]$. Then for some $p^{**} \in G$ and $q^{**} \in H$, there is a dense embedding $d \in V$ from
$\mbbQ \restrict q^{**}$ to $\mbbP \restrict p^{**}$
such that $G \restrict p^{**}$ is a filter
generated by $d[ H \restrict q^{**} ]$.
\end{aenumerate}
\end{lemma}

\begin{proof}
(1) Let $\dot{G}$ be a $\mbb{Q}$-name of $G$, and
take $q^* \in H$ which forces that $\dot{G}$ is a $\mbb{P}$-generic filter over $V$.
If $1_\mbbQ$ forces the existence of $\mbbP$-generic filter over $V$,
then take $\dot{G}$ so that $1_\mbbQ$ forces $\dot{G}$ to be
a $\mbbP$-generic filter over $V$, and let $q^* := 1_\mbbQ$.
In $V$, for each $q \in \mbb{Q} \restrict q^*$, let
\[
X_q := \{ p \in \mbb{P} \mid q \Vdash_\mbb{P} \lchon\, p \in \dot{G} \,\rchon \} \, .
\]

\begin{claim} \label{claim:X_q}
In $V$, for any $q \in \mbb{Q} \restrict q^*$,
$X_q$ has a lower bound in $\mbbP$, and $\inf X_q \in X_q$.
\end{claim}

\noindent
\emph{Proof of Claim}.
We work in $V$. Take an arbitrary $q \in \mbb{Q} \restrict q^*$.

Let $D_0$ be the set of all lower bounds of $X_q$ in $\mbb{P}$
and $D_1$ be the set of all $p' \in \mbb{P}$ which is incompatible
with some $p \in X_q$.
Then $D := D_0 \cup D_1$ is dense in $\mbb{P}$.
So $q \Vdash \lchon\, \dot{G} \cap D \neq \emptyset \,\rchon$.
But, $q \Vdash \lchon\, \dot{G} \cap D_1 = \emptyset \,\rchon$ since $q$ forces that $\dot{G}$ is a filter.
So $q \Vdash \lchon\, \dot{G} \cap D_0 \neq \emptyset \,\rchon$.
Thus $D_0 \neq \emptyset$, and $q$ forces that $\inf X_q = \sup D_0 \in \dot{G}$, that is,~$\inf X_q \in X_q$.
\hfill $\square$(Claim \ref{claim:X_q})

\vskip.5\baselineskip

For each $q \in \mbb{Q} \restrict q^*$ let $f(q) := \inf X_q$, and let $p^* := f( q^* )$.
Since $q^* \in H$, we have $X_{q^*} \subseteq \dot{G}^H = G$,
and so $p^* \in G$. It is also easy to see that
$G \restrict p^*$ is generated by $f[ H \restrict q^* ]$.
We prove that $f$ is a projection from $\mbb{Q} \restrict q^*$ to $\mbb{P} \restrict p^*$.
Clearly $f$ satisfies the conditions (i) and (ii) of the definition of a projection. We check (iii).

Suppose $q \in \mbb{Q} \restrict q^*$, $p \in \mbb{P} \restrict p^*$ and $p \leq_\mbbP f(q)$.
For a contradiction, assume there is no $q' \leq_\mbbQ q$ with $f(q') \leq_\mbbP p$.
Then there is no $q' \leq_\mbbQ q$ which forces that $p \in \dot{G}$.
Then $q \Vdash \lchon\, p \notin \dot{G} \,\rchon$.
Moreover, $q \Vdash \lchon\, f(q) \in \dot{G} \,\rchon$ by Claim \ref{claim:X_q}.
So $q \Vdash \lchon\, f(q) - p \in \dot{G} \,\rchon$, that is, $f(q) - p \in X_q$.
But $f(q) - p <_\mbbP f(q) = \inf X_q$. This contradicts the definition
of $X_q$.

\vskip.5\baselineskip

\noindent
(2) Let $p^*$, $q^*$ and $f$ be those defined in the proof of (1).
Moreover, let $\mbbP^*$, $G^*$, $\mbbQ^*$ and $H^*$ be $\mbbP \restrict p^*$,
$G \restrict p^*$, $\mbbQ \restrict q^*$ and $H \restrict q^*$, respectively.
Then $G^*$ and $H^*$ are $\mbbP^*$-generic and $\mbbQ^*$-generic over $V$, respectively.
Moreover, $f : \mbbQ^* \to \mbbP^*$ is a projection, and
$G^*$ is generated by $f[ H^* ]$.
So $H^*$ is $( \mbbQ^* / G^* )^f$-generic over $V[ G^* ]$.

Since $H \in V[G] = V[ G^* ]$, we have $H^* \in V[ G^* ]$.
Then the complement of $H^*$ is not dense in $( \mbbQ^* / G^* )^f$.
So there is $q^{**} \in H^*$ such that $( \mbbQ^* / G^* )^f \restrict q^{**} \subseteq H^*$.
Since $( \mbbQ^* / G^* )^f = f^{-1} [ G^* ]$, we have that if $q \leq_\mbbQ q^{**}$, and $f(q) \in G$,
then $q \in H$.
By extending $q^{**}$ if necessary, we may assume that $q^{**}$ forces this.
That is, in $V$, $q^{**}$ forces that if $q \leq_\mbbQ q^{**}$, and $f(q) \in \dot{G}$,
then $q \in \dot{H}$. Here $\dot{H}$ is the canonical $\mbbQ$-name for a $\mbbQ$-generic filter.
Let $p^{**} := f( q^{**} )$ and $d$ be the restriction of $f$
to $\mbbQ \restrict q^{**}$.
Then $G \restrict p^{**}$ is generated by $d[ H \restrict q^{**} ]$.

Working in $V$, we show that
$d : \mbbQ \restrict q^{**} \to \mbbP \restrict p^{**}$
is a dense embedding. Note that $d$ is a projection,
since $f$ is a projection.
So it suffices to show that for any $q_0 , q_1 \in \mbbQ \restrict q^{**}$
if $d( q_0 )$ and $d( q_1 )$ are compatible in $\mbbP$,
then $q_0$ and $q_1$ are compatible in $\mbbQ$.

Suppose $q_0 , q_1 \in \mbbQ \restrict q^{**}$,
and $d( q_0 )$ and $d( q_1 )$ are compatible in $\mbbP$.
Since $d$ is a projection, we can take $q \leq q^{**}$ with
$d(q) \leq_\mbbP d( q_0 ) , d( q_1 )$.
Then, since $q$ forces that $d(q) = f(q) \in \dot{G}$,
it also forces that $d( q_0 ) , d( q_1 ) \in \dot{G}$.
Then, since $q \leq_\mbbQ q^{**}$, $q$ forces that $q_0 , q_1 \in \dot{H}$.
So $q_0$ and $q_1$ are compatible in $\mbbQ$.
\end{proof}

For a poset $\mbb{P}$, $\mbbB ( \mbb{P} )$ is the complete Boolean algebra
consisting of all non-empty regular open subsets of $\mbb{P}$
ordered by inclusion.
Note that there is a natural dense embedding $d : \mbb{P} \to \mbbB ( \mbb{P} )$,
that is, $d(p)$ is the smallest regular open subset of $\mbb{P}$ containing $p$.
If $\mbbP$ is separative then $d$ is injective, so we regard $\mbbP$ as a suborder of $\mbbB(\mbbP)$.
We say that posets $\mbb{P}$ and $\mbb{Q}$ are \emph{forcing equivalent} if
$\mbbB ( \mbb{P} )$ and $\mbbB ( \mbb{Q} )$ are isomorphic.

Suppose that $\mbbB$ is a complete Boolean algebra and $\dot{\sigma}$ is a $\mbbB$-name for a subset of a set $x$ in $V$. Let $\mbbB(\dot{\sigma})$ denote the complete Boolean subalgebra of $\mbbB$ that is generated by the Boolean values $\bool{ y\in \dot{\sigma} }_\mbbB$ for $y \in x$.

\vskip.5\baselineskip



For a regular cardinal $\alpha$ and an ordinal $\beta > \alpha$, $\coll(\alpha, \beta)$ is the poset of all partial functions $p$ from $\alpha$ to $\beta$ such that $|p|<\alpha$, ordered by reverse inclusion.
$\coll(\alpha, {<}\beta)$ is the poset of all partial functions $p$ from $\alpha \times \beta$ to $\beta$ such that $|p| < \alpha$ and $p(\gamma, \delta)\in \delta$ for all $(\gamma, \delta)\in \dom(p)$, ordered by reverse inclusion.
For $\alpha<\gamma<\beta$,
\begin{center}
  $\coll(\alpha, [\gamma, \beta))=
  \{p \in \coll ( \alpha, {<}\beta ) \mid \dom (p) \subseteq \alpha \times [ \gamma, \beta )\}$.
\end{center}
For $G$ a $\coll(\alpha, {<}\beta)$-generic filter over $V$ and $\alpha < \gamma < \beta $, $G_\gamma$ and $G_{[\gamma, \beta)}$ denote $G \cap \coll(\alpha, {<} \gamma)$ and $G \cap \coll(\alpha, [\gamma, \beta))$, respectively.

  For a regular cardinal $\mu$, $\add(\mu, 1)$ is the poset $\mu^{< \mu}$ ordered by reverse inclusion.
  For an ordinal $\gamma$, let $\add(\mu, \gamma)$ be the ${<}\mu$-support product of $\gamma$-many copies of $\add ( \mu , 1 )$.

 The following is a standard fact.

 \begin{fact}\label{fact:cohen_equiv}
  If $\mu$ is a regular cardinal with $\mu^{<\mu}=\mu$, and $\mbb{Q}$ is a nonatomic $\mu$-closed forcing of size $\mu$, then $\mbb{Q}$
  has a dense subset that is isomorphic to a dense subset in $\add(\mu, 1)$.
 \end{fact}

\vskip.5\baselineskip

Next, we recall the definition and some basic properties of the Solovay model which was introduced by Solovay \cite{solovay1970model} as $\mu = \omega$.

\begin{definition}
  Assume $(\ast)_{\mu,\kappa}$.
  We let
  \[ V(({\mu^\mu})^{V[G]}) := \mathrm{HOD}^{V[G]}_{V \,\cup\, {(\mu^\mu)}^{V[G]}}  \]
  that is, $V(({\mu^\mu})^{V[G]})$ is the collection of sets which are hereditarily ordinal definable in $V[G]$ using parameters from $V$ and $(\mu^\mu)^{V[G]}$.
  We call $\VomegaVG$ the \textit{Solovay model} for $\kappa$.
\end{definition}

Assume $(\ast)_{\omega, \kappa}$.
Solovay \cite{solovay1970model} proved that the Solovay model $\VomegaVG$ satisfies $\ZF + \DC$, and in $\VomegaVG$ all sets of reals have regularity properties such as the Baire property, Lebesgue measurability and the perfect set property. Here we also recall some technical details about the Solovay model. 

Suppose $X \in \VomegaVG$. Then there are a formula $\varphi$, $v \in V$ and $r \in \mbbR^{V[G]}$ such that
\[
x \in X \ \Leftrightarrow \ V[G] \models \varphi (v,r,x) \, .
\]
We may assume that $\kappa$ is definable from $v$.

The following fact is one of the keys in the analysis of the Solovay model.
We later consider the generalized statement (Proposition \ref{prop:coll_unctble}). 

\begin{fact}\label{fact:factor}
  Assume $(\ast)_{\omega, \kappa}$.
  Suppose $x \in \mcalH_\kappa^{V[G]}$.
  Then, in $V[G]$, there is a $\coll(\omega, {<}\kappa)$-generic filter $G'$ over $V[x]$ such that $V[G] = V[x][G']$.
\end{fact}

Let $\psi (v,r,x)$ be the formula stating that $\Vdash_{\coll(\omega, {<}\kappa)} \varphi ( \check{v} , \check{r} , \check{x} )$.
If $\mbbP$ is a poset in $\mcalH_\kappa$, and $h$ is a $\mbbP$-generic filter over $V$ with $h \in V[G]$ and $r \in V[h]$, then for any $x \in V[h]$ we have
\[
x \in X \ \Leftrightarrow \ V[h] \models \psi (v,r,x) \, .
\]
This is because of Fact \ref{fact:factor} and the homogeneity of $\coll(\omega, {<}\kappa)$.

Also, the following is a standard fact about $\mathrm{HOD}$.
\begin{fact}\label{fact:surj}
  Assume $(\ast)_{\mu, \kappa}$.
  Suppose $X \in V((\mu^\mu)^{V[G]})$ is a non-empty set.
  Then for some ordinal $\lambda$ there is a surjection $f \in V((\mu^\mu)^{V[G]})$ from $V_\lambda \times (\mu^\mu)^{V[G]}$ to $X$, where $V_\lambda$ denotes the one defined in $V$.
\end{fact}

\section{Perfect set dichotomy in Solovay model}



In this section, we prove that in the Solovay model $\VomegaVG$
the perfect set dichotomy holds for all equivalence relations on $\mbbR$.

\begin{theorem}\label{dic_equi_Solovay}
  Assume $(\ast)_{\omega, \kappa}$.
  In $V((\omega^\omega)^{V[G]})$, suppose $E$ is an equivalence relation on $\mbbR$.
  Then either $\mbbR/E$ is well-orderable, or else there is a perfect set of pairwise $E$-inequivalent reals.
\end{theorem}

\begin{proof}
  We work in $V[G]$.
  Let $E$ be an equivalence relation on $\mbbR$ in $\VomegaVG$.
  Below, $\mcalH_\kappa$ denotes the one in $V$.
  Take a formula $\varphi$, $v \in V$ and $r \in \omega^\omega$ such that $x \, E \, y$ if and only if $\varphi (v,r,x,y)$ for all $x,y \in \mbbR$.
  By replacing the ground model $V$ with $V[r]$, we may assume that $r \in V$.
  So we may omit $r$, that is, for all $x,y \in \mbbR$
  \[
  x \, E \, y \ \Leftrightarrow \ \varphi (v,x,y) \, .
  \]
  We may also assume $\kappa$ is definable from $v$.
  
  Let $\psi (v,x,y)$ be the formula stating that $\Vdash_{\coll(\omega, {<}\kappa)} \varphi ( \check{v} , \check{x} , \check{y} )$.
  If $\mbbP$ is a poset in $\mcalH_\kappa$, and $h$ is a $\mbbP$-generic filter over $V$ with $h \in V[G]$, then for any $x,y \in V[h]$ we have
  \[
  x \, E \, y \ \Leftrightarrow \ V[h] \models \psi (v,x,y) \, .
  \]

  In $V$, let $\Omega$ be the set of all triples $( \xi , p , \dot{x} )$ such that $\xi < \kappa$, $p \in \coll (\omega, \xi)$ and $\dot{x}$ is a 
  $\coll (\omega, \xi)$-name for a real with $\dot{x} \in \mcalH_\kappa$.
  For $( \xi , p , \dot{x} ) \in \Omega$, let
  \[
  \Eval (\xi,p,\dot{x}) :=
  \{ \dot{x}[h] \mid \mbox{$h$ is a 
  $\coll (\omega, \xi)$-generic filter over $V$ with $p\in h$} \} \, .
  \]

  A non-empty set $X\subseteq\mbbR$ is called an $E$-\textit{component} if $x \, E \, y$ for all $x, y \in X$. The rest of the proof splits into two cases.

  \vskip.5\baselineskip

  \noindent\underline{(Case I)} For any $( \xi , q , \dot{x} ) \in \Omega$ there is $p \leq q$ such that $\Eval (\xi,p,\dot{x})$ is an $E$-component.
  
  \vskip.5\baselineskip

  First we establish the following claim.

  \begin{claim} \label{E_component}
  For any $a \in \mbbR$ there is $( \xi , p , \dot{x} ) \in \Omega$ such that $a \in \Eval ( \xi , p , \dot{x} )$ and $\Eval ( \xi , p , \dot{x} )$ is an $E$-component.
  \end{claim}

  \noindent\textit{Proof of Claim}.
  Suppose $a \in \mbbR$. Then there is $\xi < \kappa$ and a $\coll (\omega, \xi)$-generic filter $h$ over $V$ such that $a = \dot{x}[h]$. Let
  \[
  D :=
  \{ p \in \coll ( \omega , \xi ) \mid \mbox{$\Eval ( \xi , p , \dot{x} )$ is an $E$-component} \} \, .
  \]
  Then $D$ is dense in $\coll ( \omega , \xi )$ by the assumption of Case I.
  By the homogeneity of $\coll(\omega, {<}\kappa)$, $p \in D$ if and only if in $V$, $1_{\coll(\omega, {<}\kappa)}$ forces the following.
  \begin{renumerate}
      \item $\varphi ( \check{v} , \check{\dot{x}} [ h_0 ] , \check{\dot{x}} [ h_1 ] )$ for any $\coll ( \omega , \check{\xi} )$-generic filters $h_0$ and $h_1$ over $V$ containing $\check{p}$.
  \end{renumerate}
  Here note that $\check{\dot{x}}$ is a $\coll(\omega, {<}\kappa)$-name which will be interpreted as a $\coll ( \omega , \xi )$-name $\dot{x} \in V$ in generic extensions of $V$ by $\coll(\omega, {<}\kappa)$, and so all parameters in (i) are check names.
  So $D \in V$.
  
  Take $p \in h \cap D$. Then $( \xi , p , \dot{x} )$ is as desired.
  \hfill $\square$(Claim \ref{E_component})

  \vskip.5\baselineskip

  Since $\AC$ holds in $V$, we can take a well-ordering $\leq_\Omega$ of $\Omega$.
  By Claim \ref{E_component}, for each $A \in \mbbR / E$, there is $( \xi , p , \dot{x} ) \in \Omega$ such that $\Eval ( \xi , p , \dot{x} ) \subseteq A$. For $A \in \mbbR / E$ let $( \xi_A , p_A , \dot{x}_A )$ be the $\leq_\Omega$-least such $( \xi , p , \dot{x} ) \in \Omega$.
  For $A,B \in \mbbR / E$, let $A \trianglelefteq B$ if $( \xi_A , p_A , \dot{x}_A ) \leq_\Omega ( \xi_B , p_B , \dot{x}_B )$.
  Then $\trianglelefteq$ well-orders $\mbbR / E$.
  Since $\leq_\Omega$ and $ \Eval $ belong to $ \VomegaVG$, $\trianglelefteq$  belongs to $\VomegaVG$.

  \vskip.5\baselineskip
  
  \noindent\underline{(Case II)} There is $( \xi , q , \dot{x} ) \in \Omega$ such that $\Eval ( \xi , p , \dot{x} )$ is not an $E$-component for any $p \leq q$.

  \vskip.5\baselineskip
 
  Take $(\xi, q, \dot{x}) \in \Omega$ as in the assumption.
  Let $\dot{x}_\mathrm{left}$ and $\dot{x}_\mathrm{right}$ be $\coll(\omega,\xi)\times \coll(\omega, \xi)$-names in $V$ such that  $\dot{x}_\mathrm{left} [ h_0 \times h_1 ] = \dot{x} [ h_0 ]$, and $\dot{x}_\mathrm{right} [ h_0 \times h_1 ] = \dot{x} [ h_1 ]$ for any $\coll ( \omega , \xi ) \times \coll ( \omega , \xi )$-generic filter $h_0 \times h_1$ over $V$.
 
  \begin{claim} \label{claim:product_force}
  In $V$, $(q,q) \Vdash_{\coll(\omega, \xi) \times \coll(\omega, \xi)} \lnot \psi ( \check{v} , \dot{x}_\mathrm{left} , \dot{x}_\mathrm{right} )$.
  \end{claim}
 
  \noindent\textit{Proof of Claim}.
  Assume the claim fails. Then there is $(q_0,q_1)\leq (q,q)$ such that $(q_0,q_1) \Vdash \psi ( \check{v} , \dot{x}_\mathrm{left} , \dot{x}_\mathrm{right} )$ in $V$.
  
  By the choice of $( \xi , q , \dot{x} )$, $\Eval ( \xi, q_0, \dot{x} )$ is not an $E$-component.
  Hence there are $\coll(\omega,\xi)$-generic filters $h_0, h_1$ over $V$ such that $q_0 \in h_0, h_1$ and $\lnot ( \dot{x} [ h_0 ] \, E \, \dot{x} [ h_1 ] )$.  
  Since $\mcalP^{V[h_i]} (\coll(\omega, \xi))$ is countable, we can take a $\coll(\omega, \xi)$-generic filter $k$ over $V[h_0]$ and $V[h_1]$ with $q_1 \in k$.

  By the choice of $( q_0 , q_1 )$, $\psi ( v , \dot{x}_\mathrm{left} , \dot{x}_\mathrm{right} )$ holds in $V[ h_i \times k ]$ for $i = 0,1$.
  This means that $\dot{x}_\mathrm{left} [ h_i \times k ] \, E \, \dot{x}_\mathrm{right} [ h_i \times k ]$ for $i = 0,1$.
  Then $\dot{x} [ h_0 ] \, E \, \dot{x} [k] \, E \, \dot{x} [ h_1 ]$.
  This contradicts the choice of $h_0$ and $h_1$.
  \hfill $\square$(Claim \ref{claim:product_force})

  \vskip.5\baselineskip
 
  Note that $\mcalP^V (\coll(\omega, \xi))$ is countable in $\VomegaVG$.
  Then in $\VomegaVG$ we can construct a sequence $\langle q_s \mid s \in 2^{< \omega} \rangle$ in $\coll ( \omega , \xi )$ such that
  \begin{itemize}
      \item $q_\emptyset = q$, and $q_s \geq q_t$ if $s \subseteq t$,
      \item $h_b \times h_c$ is a $\coll ( \omega , \xi ) \times \coll ( \omega , \xi )$-generic filter over $V$ for any distinct $b,c \in 2^\omega$,
  \end{itemize}
  where $h_b$ is a filter generated by $\{ q_{b \upharpoonright n} \mid n < \omega \}$ for $b \in 2^\omega$.
  By Claim \ref{claim:product_force}, $\langle \dot{x} [ h_b ] \mid b \in 2^\omega \rangle$ is pairwise $E$-inequivalent.
  By the perfect set property, we can take a perfect subset of $\{ \dot{x} [ h_b ] \mid b \in 2^\omega \}$ in $\VomegaVG$.
  This set is a perfect set of $E$-inequivalent reals.
  \end{proof}

 
 


\section{Generalization of Solovay model}\label{sec:higher}

We consider the generalized Solovay model $V((\mu^\mu)^{V[G]})$ under the assumption $(\ast)_{\mu, \kappa}$ for $\mu>\omega$
and show that the perfect set dichotomy holds for all equivalence relations on $\mu^\mu$ in this model.

First, we recall basic definitions.
Let $\mu$ be a regular uncountable cardinal. $\mu^\mu$ is equipped with the topology generated by the basic open sets 
  $N_t = \{ x\in \mu^\mu \mid t \subseteq x\}$
 for each $t \in \mu^{<\mu}$.

  Suppose that $T$ is a subtree of $\mu^{<\mu}$, 
  that is a $\subseteq$-downward closed subset of $\mu^{<\mu}$.
  A \textit{branch} of $T$ is $b \in \mu^\mu$ such that 
    $b \restrict \alpha \in T$
    for all $\alpha < \mu$.
    The \textit{body} $[T]$ of $T$ is the collection of branches in $T$.
  A node $s \in T$ is \textit{terminal} if it has no direct successor in $T$ and \textit{splitting} if it has at least two direct successors in $T$.
  $T$ is \textit{closed} if every strictly increasing sequence in $T$ of length $<\mu$ has an upper bound in $T$.
  $T$ is \textit{perfect} if $T$ is closed and for each node $s \in T$ there is a splitting node $t \in T$ with $s \subseteq t$.
  
  A subset $A \subseteq \mu^\mu$ is \textit{perfect} if $A=[T]$ for some perfect tree $T$.
  A subset $A \subseteq \mu^\mu$ has the \textit{perfect set property} if either $|A|\leq\mu$ or $A$ has a perfect subset.

\subsection{Quotient forcing of Lévy collapse}
 The key reason why Solovay's  original proof cannot be straightforwardly generalized to the uncountable setting is that a quotient forcing of a $\mu$-closed forcing is not necessarily $\mu$-closed. 
 Consequently, the higher analogue of Fact \ref{fact:factor} does not hold.
 To overcome this problem, we focus only on the extensions of $V$ of which $V[G]$ is an extension by the Lévy collapse.

\begin{definition}\label{def:nice}
  Assume $(\ast)_{\mu,\kappa}$.
  We say that a set of ordinals $x \in \mcalH_\kappa^{V[G]}$ 
  is \textit{nice} if for any 
  $y \in \mcalH_\kappa^{V[G]}$
  there is a poset $\mbbP \in \mcalH_\kappa^{V[x]}$
  and a $\mbbP$-generic filter $g$ over $V[x]$ such that $\mbbP$ is $\mu$-closed in $V[x]$, 
  and $y \in V[x][g]$.
\end{definition}

Note that the niceness is definable in $\VmuVG$.
Note that for each $\xi < \kappa$, (a set of ordinals naturally coding) $G_\xi$ is nice, since $V[G]$ is a forcing extension of $V[ G_\xi ]$ by a $\mu$-closed forcing $\coll ( \mu , [ \xi , \kappa ))$.

We prove the following generalization of Fact \ref{fact:factor}.
Note that the following proposition for $\mu = \omega$ is the same as Fact \ref{fact:factor}.

\begin{prop} \label{prop:coll_unctble}
  Assume $(\ast)_{\mu, \kappa}$.
Suppose $x \in \mcalH_\kappa^{V[G]}$ is nice in $V[G]$.
Then, in $V[G]$, there is a $\coll ( \mu , {<} \kappa )$-generic filter $G'$ over $V[x]$ such that $V[G] = V[x][G']$.
\end{prop}

To prove the proposition, we use the notion of the (strong) $\mu$-strategic closure of posets.
Let $\mbb{P}$ be a poset and $\mu$ be a regular cardinal.

$\Game_\mrm{I} ( \mbb{P} , \mu )$ is the following two players game of I and II of length $\mu$.
\[
\begin{array}{c|ccccccccccc}
\mbox{I} & p_0 &  & p_2 & & \dots & p_\omega & & \dots & p_{\omega + \omega} &  & \cdots \\
\hline
\mbox{II} & & p_1 & & p_3 & \dots &  & p_{\omega +1} & \dots &  & p_{\omega + \omega + 1} & \cdots
\end{array}
\]
For each $\xi < \mu$, I and II in turn choose a lower bound $p_\xi$ of $\{ p_\eta \mid \eta < \xi \}$ in $\mbb{P}$ if exists.
I chooses $p_\xi$ for each even $\xi$, and II chooses $p_\xi$ for each odd $\xi$.
For a limit $\xi < \mu$, if $\{ p_\eta \mid \eta < \xi \}$ has no lower bound, then the play is over at this point, and I wins.
II wins if a play has been continued to construct a descending sequence $\langle p_\xi \mid \xi < \mu \rangle$.

$\Game_\mrm{II} ( \mbb{P} , \mu )$ is the following modification of $\Game_\mrm{I} ( \mbb{P} , \mu )$.
\[
\begin{array}{c|ccccccccccc}
\mbox{I} & & p_1 &  & p_3 & \dots & & p_{\omega + 1} & \dots & & p_{\omega + \omega +1} & \cdots \\
\hline
\mbox{II} & p_0 = 1_\mbb{P} & & p_2 & &  \dots & p_\omega & & \dots & p_{\omega + \omega} & & \cdots
\end{array}
\]
As in $\Game_\mrm{I} ( \mbb{P} , \mu )$, I and II in turn choose a lower bound $p_\xi$ of $\{ p_\eta \mid \eta < \xi \}$.
I chooses $p_\xi$ for each odd $\xi$, and II chooses $p_\xi$ for each even $\xi$.
But $p_0$ must be $1_\mbb{P}$.
For a limit $\xi < \mu$, if $\{ p_\eta \mid \eta < \xi \}$ has no lower bound, then the play is over at this point, and I wins.
II wins if a play has been continued to construct a descending sequence $\langle p_\xi \mid \xi < \mu \rangle$.

We say that $\mbb{P}$ is (resp.~\emph{strongly}) $\mu$-\emph{strategically closed}
if II has a winning strategy in $\Game_\mrm{II} ( \mbb{P} , \mu )$ (resp.~$\Game_\mrm{I} ( \mbb{P} , \mu )$).
Note that $\Game_\mrm{I} ( \mbb{P} , \mu )$ is harder than $\Game_\mrm{II} ( \mbb{P} , \mu )$ for II.
So the strong $\mu$-strategic closure implies the $\mu$-strategic closure.
Note also that if $\mbb{P}$ is $\mu$-closed, then $\mbb{P}$ is strongly $\mu$-strategically closed
since II wins $\Game_\mrm{I} ( \mbb{P} , \mu )$ no matter how he moves (legally).
We also use the fact that if $\mbb{P}$ is (resp.~strongly) $\mu$-strategically closed,
then $\mbb{P} \restrict p$ is (resp.~strongly) $\mu$-strategically closed for any $p \in \mbb{P}$.

In the proof of Proposition \ref{prop:coll_unctble}, we want to use the following fact
which slightly strengthen the well-known fact due to McAloon.
In the following fact, we cannot weaken the assumption (i) to the $\mu$-strategic closure
of $\mbbP$.

\begin{fact}[Eskew \cite{eskew2025comparing}] \label{fact:coll_equivalent}
Let $\mu$ and $\nu$ be regular cardinals with $\mu \leq \nu$.
Suppose $\mbb{P}$ is a poset with the following properties.
\begin{renumerate}
\item $\mbb{P}$ is strongly $\mu$-strategically closed.
\item $| \mbb{P} | = \nu$, and $\Vdash_\mbb{P} \lchon\, | \nu | = \mu \,\rchon$.
\end{renumerate}
Then $\mbb{P}$ is forcing equivalent to $\coll ( \mu , \nu )$.
\end{fact}

We need more preliminaries on the strategic closure.
The following fact is well-known, and its proof is straightforward.

\begin{fact}[folklore] \label{fact:str_closed_basic}
Let $\mu$ be a regular cardinal and $\mbb{P}$ and $\mbb{Q}$ be posets.
\begin{aenumerate}
\item Suppose there is a dense embedding from $\mbb{P}$ to $\mbb{Q}$.
Then $\mbb{P}$ is (resp.~strongly) $\mu$-strategically closed
if and only if $\mbb{Q}$ is (resp.~strongly) $\mu$-strategically closed.
\item Suppose there is a projection from $\mbb{Q}$ to $\mbb{P}$.
If $\mbb{Q}$ is $\mu$-strategically closed, then $\mbb{P}$ is $\mu$-strategically closed.
\end{aenumerate}
\end{fact}

(2) does not hold for the strong $\mu$-strategic closure.
For example, suppose $\mu \geq \omega_2$ and let $\mbbP$ be
the poset adding a non-reflecting stationary subset of
$E^\mu_\omega = \{ \alpha < \mu \mid \mathrm{cf} ( \alpha ) = \omega \}$.
In $V^\mbbP$, let $\dot{\mbbS}$ be the poset shooting a club through the complement of
the non-reflecting stationary set added by $\mbbP$.
Then $\mbbQ := \mbbP * \dot{\mbbS}$ has a $\mu$-closed dense subset,
and so it is strongly $\mu$-strategically closed.
Also, there is a projection from $\mbbQ$ to $\mbbP$.
But $\mbbP$ is not strongly $\mu$-strategically closed.

As for the strong $\mu$-strategic closure, we have the following.

\begin{lemma} \label{lem:strong_str_closed_quotient}
Let $\mu$ be a regular cardinal, $\mbb{P}$ be a poset and $\dot{\mbbS}$ be a $\mbb{P}$-name for a poset.
Suppose $\mbb{P} * \dot{\mbb{S}}$ is strongly $\mu$-strategically closed, and $1_\mbb{P}$ forces
$\dot{\mbbS}$ to be $\mu$-strategically closed.
Then $\mbb{P}$ is strongly $\mu$-strategically closed.
\end{lemma}

\begin{proof}
Let $\tau$ be a winning strategy of II for $\Game_\mrm{I} ( \mbb{P} * \dot{\mbbS} , \mu )$
and $\dot{\sigma}$ be a $\mbb{P}$-name for a winning strategy of II for $\Game_\mrm{II} ( \dot{\mbbS} , \mu )$.

We describe a strategy of II for $\Game_\mrm{I} ( \mbb{P} , \mu )$.
Let $p_\xi$ denote the $\xi$-th move in this game.
For each $\xi$, II auxiliary chooses a $\mbb{P}$-name $\dot{s}_\xi$ so that
\begin{renumerate}
\item $( p_\xi , \dot{s}_\xi ) \in \mbb{P} * \dot{\mbbS}$,
and $( p_\xi , \dot{s}_\xi ) \leq ( p_\eta , \dot{s}_\eta )$ for all $\eta < \xi$,
\item $( p_\xi , \dot{s}_\xi ) = \tau ( \langle ( p_\eta , \dot{s}_\eta ) \mid \eta < \xi \rangle )$ if $\xi$ is odd,
\item $p_\xi \Vdash_\mbb{P} \lchon\, \dot{s}_\xi = \dot{\sigma} ( \langle \dot{s}_\eta \mid \eta < \xi \rangle ) \,\rchon$
if $\xi$ is even.
\end{renumerate}

Suppose $\xi$ is an odd ordinal $< \mu$, and I has chosen $p_{\xi - 1}$.
We also assume that $\dot{s}_\eta$ has been chosen for each $\eta < \xi -1$ as mentioned above.
Then II chooses $p_\xi \in \mbb{P}$ together with $\dot{s}_{\xi - 1} , \dot{s}_\xi$ as follows.
Note that $\xi -1$ is even.
Note also that, since $p_{\xi -1}$ is a lower bound of $\{ p_\eta \mid \eta < \xi -1 \}$,
it forces $\langle \dot{s}_\eta \mid \eta < \xi -1 \rangle$ to be moves in $\Game_\mrm{II} ( \dot{\mbbS} , \mu )$
in which II has moved according to a winning strategy $\dot{\sigma}$.
So $p_{\xi -1}$ forces that $\{ \dot{s}_\eta \mid \eta < \xi -1 \}$ has a lower bound in $\dot{\mbbS}$.
Let $\dot{s}_{\xi -1}$ be a $\mbb{P}$-name such that
\[
p_{\xi - 1} \Vdash_\mbb{P} \lchon\, \dot{s}_{\xi -1} = \dot{\sigma} ( \langle \dot{s}_\eta \mid \eta < \xi -1 \rangle ) \,\rchon \, .
\]
Then $p_{\xi -1}$ and $\dot{s}_{\xi -1}$ satisfy (i) -- (iii) above.
Next, let $p_\xi$ and $\dot{s}_\xi$ be such that
\[
( p_\xi , \dot{s}_\xi ) = \sigma ( \langle ( p_\eta , \dot{s}_\eta ) \mid \eta < \xi \rangle ) \, .
\]
Then $p_\xi$ and $\dot{s}_\xi$ satisfy (i) -- (iii) clearly.

Note that the above strategy of II for $\Game_\mrm{I} ( \mbb{P} , \mu )$ is a winning strategy:
Suppose $\xi$ is a limit ordinal, and $\langle p_\eta \mid \eta < \xi \rangle$ is an initial play of $\Game_\mrm{I} ( \mbb{P} , \mu )$
in which II has moved according to the above strategy.
For each $\eta < \xi$, let $\dot{s}_\eta$ be a $\mbb{P}$-name auxiliary chosen by II.
Then, by (ii), $\{ ( p_\eta , \dot{s}_\eta ) \mid \eta < \xi \}$ has a lower bound $( p , \dot{s} )$ in $\mbb{P} * \dot{\mbbS}$.
Then $p$ is a lower bound of $\{ p_\eta \mid \eta < \xi \}$ in $\mbb{P}$.
\end{proof}

Using the above lemma, we can prove the following.

\begin{lemma} \label{lem:4_models}
Suppose $W_0$, $W_1$, $W_2$ and $W_3$ are transitive models of $\ZFC$
with $W_0 \subseteq W_1 \subseteq W_2 \subseteq W_3$, and for each $i,j$ with $i < j < 4$,
$W_j$ is a forcing extension of $W_i$ by a poset $\mbbP_{ij} \in W_i$.
Let $\mu$ be a regular cardinal in $W_3$, and assume the following.
\begin{renumerate}
\item $\mbbP_{02}$ is strongly $\mu$-strategically closed in $W_0$.
\item $\mbbP_{13}$ is $\mu$-strategically closed in $W_1$.
\end{renumerate}
Let $G_{01}$ be a $\mbbP_{01}$-generic filter over $W_0$ with $W_1 = W_0 [ G_{01} ]$.
Then there is $p \in G_{01}$ such that $\mbbP_{01} \restrict p$ is strongly $\mu$-strategically closed.
\end{lemma}

\begin{proof}
We may assume each $\mbbP_{ij}$, except for $\mbbP_{01}$, is a complete Boolean algebra in $W_i$
by Fact \ref{fact:str_closed_basic} (1).
For each $i,j$ with $i < j < 4$ let $G_{ij}$ be a $\mbbP_{ij}$-generic filter over $W_i$
with $W_j = W_i [ G_{ij} ]$.

Note that $W_1 [ G_{12} ] = W_2 \subseteq W_3 = W_1 [ G_{13} ]$.
Then, by Lemma \ref{lem:projection} (1),
by restricting $\mbbP_{12}$ and $\mbbP_{13}$ to lower bounds of some conditions,
we may assume that in $W_1$ there is a projection from $\mbbP_{13}$ to $\mbbP_{12}$.
Then $\mbbP_{12}$ is $\mu$-strategically closed in $W_1$
by (ii) and Fact \ref{fact:str_closed_basic} (2).

Let $\dot{\mbbP}_{12}$ be a $\mbbP_{01}$-name of $\dot{\mbbP}_{12}$,
which $1_{\mbbP_{01}}$ forces to be a $\mu$-strategically closed poset.
Then $G_{01} * G_{12}$ is a $\mbbP_{01} * \dot{\mbbP}_{12}$-generic filter over $W_0$,
and $W_0 [ G_{01} * G_{12} ] = W_2 = W_0 [ G_{02} ]$.
By Lemma \ref{lem:projection} (2), we can take $p * \dot{q} \in G_{01} * G_{12}$
and $r \in G_{02}$ such that there is a dense embedding from
$\mbbP := ( \mbbP_{01} * \dot{\mbbP}_{12} ) \restrict ( p * \dot{q} )$ to $\mbbP_{02} \restrict r$ in $W_0$.
Then $\mbbP$ is strongly $\mu$-strategically closed by (i) and Fact \ref{fact:str_closed_basic} (1).
But $\mbbP = ( \mbbP_{01} \restrict p ) * ( \dot{\mbbP}_{12} \restrict \dot{q} )$,
and $p$ forces that $\dot{\mbbP}_{12} \restrict \dot{q}$ is $\mu$-strategically closed.
So $\mbbP_{01} \restrict p$ is strongly $\mu$-strategically closed
by Lemma \ref{lem:strong_str_closed_quotient}. Then $p$ is as desired.
\end{proof}

The following lemma implies 
Proposition \ref{prop:coll_unctble} easily.

\begin{lemma}\label{lem:nu+1}
 Assume $(\star)_{\mu, \kappa}$.
 Suppose $x \in \mcalH_\kappa^{V[G]}$ is nice in $V[G]$.
 Let $\nu \in (\mu, \kappa)$ be a regular cardinal with $\nu^{<\mu}=\nu$ in $V[x]$ and $x \in V[G_{\nu+1}]$.
 Then, $V[G_{\nu+1}]$ is an extension of $V[x]$ by $\coll(\mu, {<}\ \nu+1)$.
\end{lemma}

\begin{proof}
    We may assume $x$ is an $\mbb{S}$-generic filter over $V$ for some $\mbb{S} \in \mcalH_\kappa$
which is a complete Boolean algebra in $V$.
Note that $1_{\coll ( \mu , {<} \nu + 1 )}$ forces the existence
of $\mbbS$-generic filter over $V$ by the homogeneity of
$\coll ( \mu , {<}\  \nu + 1 )$.
Then, by Lemma \ref{lem:projection} (1),
there are $s \in x$ and a projection
$f \in V$ from $\coll(\mu, {<}\ \nu + 1)$
to $\mbb{S} \restrict s$ such that $x \restrict s$
is generated by $f[ G_{\nu + 1} ]$.
By replacing $\mbbS$ and $x$ with $\mbbS \restrict s$ and $x \restrict s$
if necessary, we may assume $s = 1_\mbbS$.

In $V[x]$ let
\[
\mbb{P} := ( \coll(\mu, {<}\  \nu + 1 ) \, / \, x )^f
\, .
\]
Then we have the following.
\begin{renumerate}
\item $G_{\nu + 1}$ is a $\mbbP$-generic filter over $V[x]$, and
$V[ G_{\nu + 1} ] = V[x][ G_{\nu + 1} ]$.
\item $| \mbb{P} | = \nu$ in $V[x]$.
\end{renumerate}
(ii) is because $\mbb{P} \subseteq \coll(\mu, {<}\  \nu + 1 )$,
and $| \coll(\mu, {<}\  \nu + 1) | = \nu^{< \mu} = \nu$ in $V[x]$.

We claim the following.
\begin{renumerate}
\addtocounter{enumi}{2}
\item $\mbbP \restrict p$ is strongly $\mu$-strategically closed
in $V[x]$ for some $p \in G_{\nu +1}$.
\end{renumerate}
Since $x$ is nice, there is a poset $\mbb{Q} \in \mcalH_\kappa^{V[x]}$
and a $\mbb{Q}$-generic filter $H$ over $V[x]$ such that
$\mbb{Q}$ is $\mu$-closed in $V[x]$,
and $G_{\nu + 1} \in V[x][H]$.
Then, if we let
\begin{gather*}
W_0 := V[x] \, , \ \ W_1 := V[ G_{\nu + 1} ] \, , \ \ W_2 := V[x][H] \, , \ \ W_3 := V[G] \, , \\
\mbbP_{01} := \mbbP \, , \ \ G_{01} := G_{\nu + 1} \, ,
\end{gather*}
the assumptions of Lemma \ref{lem:4_models} are satisfied.
So (iii) holds by Lemma \ref{lem:4_models}.

Take $p$ as (iii).
By (i), $V[ G_{\nu + 1} ]$ is an extension of $V[x]$
by $G_{\nu + 1} \restrict p \subseteq \mbb{P} \restrict p$.
Here note that, by (ii), (iii) and Fact \ref{fact:coll_equivalent},
\[
\mbb{P} \restrict p ~ \sim ~ \coll ( \mu , \nu ) ~ \sim ~
\coll ( \mu , {<}\  \nu + 1 )
\]
in $V[x]$, where $\sim$ is the forcing equivalence.
So $V[ G_{\nu + 1} ]$ is an extension of $V[x]$ by $\coll ( \mu , {<}\ \nu + 1 )$.
\end{proof}
\noindent\textit{Proof of Proposition 4.2}.
Suppose $x \in \mathcal{H}_\kappa^{V[G]}$ is nice in $V[G]$.
Take $\nu < \kappa$ such that $x \in V[G_{\nu+1}]$ and $\nu$ is a regular cardinal with $\nu^{<\mu}=\nu$ in $V[x]$.
By Lemma \ref{lem:nu+1}, $V[G_{\nu+1}]$ is an extension of $V[x]$ by $\coll(\mu, {<}\  \nu+1)$.
Since $V[G]$ is an extension of $V[G_{\nu+1}]$ by $\coll(\mu, [\nu+1, \kappa))$ and 
$\coll(\mu, {<}\ \nu+1) \times \coll(\mu, [\nu+1, \kappa))$ is equivalent to $\coll(\mu, {<}\kappa)$, 
we can take a $\coll(\mu, {<}\kappa)$-generic filter $G'$ over $V[x]$ such that $V[G]=V[x][G']$.
  \hfill$\square$

\subsection{Perfect set property in the generalized Solovay model}\label{sec:psp_genSolovay}

Schlicht \cite{MR3743612} proved the following fact.

\begin{fact}[Theorem 2.19 in \cite{MR3743612}]\label{fact:psp}
  Assume $(\ast)_{\mu, \kappa}$ and $\omega<\mu$.
  Then in $V[G]$, every subset of $\mu^\mu$ that is definable from an element of ${}^\mu V$ has the perfect set property.
\end{fact}

We will use this fact and a part of its proof.
Here we briefly review its proof.
The following poset plays a central role in the proof of Fact \ref{fact:psp} in \cite{MR3743612}.

 \begin{definition}\label{def:perfect_tree}
  Let $\mbb{P}_\mu$ denote the set of pairs $(t, s)$ such that
  \begin{aenumerate}
    \item $t \subseteq \mu^{<\mu}$ is a tree of size less than $\mu$.
    \item every node $u\in t$ has at most two direct successors in $t$.
    \item $s \subseteq t$ and if $u\in t$ is non-terminal in $t$, then $u\in s$ if and only if $u$ has exactly one direct successor in $t$. 
  \end{aenumerate}
  Let $(t, s) \leq_{\mbb{P}_\mu} (u, v)$ if $u\subseteq t$ and $s \cap u =v$.
 \end{definition}

 $\mbb{P}_\mu$ is a separative poset which adds a perfect binary splitting subtree of $\mu^{<\mu}$.
 For $G$ a $\mbbP_\mu$-generic filter over $V$,
 we write $T_G = \bigcup_{(t,s)\in G} t$.
 By Fact \ref{fact:cohen_equiv}, if $\mu^{<\mu}=\mu$ then $\mbb{P}_\mu$ is forcing equivalent to $\add(\mu, 1)$.

 The proof in \cite{MR3743612} of Fact \ref{fact:psp} contains a minor gap in the last part. We see the outline of a fixed proof.
 We need the following fact, which slightly generalizes Lemma 2.18 in \cite{MR3743612}. This fact can be established by almost the same argument as the original one.

\begin{fact}\label{lem:gen2.18}
  Suppose $\mbbS$ is a separative $\mu$-closed poset and $\gamma < \mu$.
  Let $G \times H$ be a $\mbbP_\mu \times \mbbS$-generic filter over $V$.
  Then, in $V[G \times H]$,
   each sequence $\langle x_i \mid i < \gamma \rangle$ of distinct branches of $T_G$ is $\add (\mu, \gamma)$-generic over $V[H]$ and $V[G \times H]$ is an extension of $V[\langle x_i \mid i < \gamma \rangle]$ by a $\mu$-closed forcing.
 \end{fact}

 \noindent\textit{Sketch of Proof}.
 We omit the details here and see only an outline.

 By Lemma 2.9 in \cite{MR3743612},
 $\langle x_i \mid i<\gamma \rangle$ is $\add(\mu, \gamma)$-generic over $V$.
 
 Let $\dot{\sigma}$ be a $\mbbP_\mu \times \mbbS$-name for a sequence in $[T_G]^{V[G][H]}$ of length $\gamma$. 
 Although $\dot{\sigma}$ is in fact a name of an element of $(\mu^\mu)^\gamma$,
 considering it as a name of an element of $\mu^{\mu \times \gamma}$,
 we can define $\mbbB(\dot{\sigma})$.
 Let $\dot{G}_{\dot{\sigma}}$ be the canonical $\mbbB(\dot{\sigma})$-name for a $\mbbB(\dot{\sigma})$-generic filter.  
 We show that $\mbbB(\mbbP_\mu \times\mbbS) / \dot{G}_{\dot{\sigma}}$ is forcing equivalent to a $\mu$-closed forcing.

 Let $\dot{b}_i$ be a $\mbb{P}_\mu$-name for $\dot{\sigma}(i)$.
 For a $\mbbP_\mu$-name $\dot{b}$ for a branch of $T_G$ and 
 $(p, r) \in \mbbP_\mu \times \mbbS$, let
 \begin{center}
 $\dot{b}_{(p, r)} := \{(\alpha, \beta) \in \Ord\times \Ord \mid (p, r) \Vdash \dot{b}(\alpha)=\beta\}$
 \end{center}
 and $\dot{\sigma}_{(p,r)} = \langle \dot{b}_{i_{(p ,r)}} \mid i <\gamma  \rangle$.

  Let ${(\mbbP_\mu \times \mbbS)}^*$ denote the set of $(p, r) = ( (t, s ), r)\in \mbbP_\mu \times \mbbS$ 
  such that $l(t)$ is a limit ordinal and 
  $\dom({\dot{b}_{i_{(p ,r)}}}) = \height(t)$ for each1 $i<\gamma$.
 Then $(\mbbP_\mu \times \mbbS)^*$ is separative and dense in $\mbbP_\mu \times \mbbS$. 
  Let
  \begin{center}
      $\mbbQ_0 = \{ (\dot{\sigma}_{({p}, r)}, \mathbbm{1}_{\mbbP}, {\mathbbm{1}}_{\mbbS} ) \mid 
      (p, r) \in ( {\mbbP_{\mu}} \times \mbbS)^* \}$,

      $\mbbQ_1 = \{ (\dot{\sigma}_{(p, r)}, p, r ) \mid 
      (p, r) \in ( {\mbbP_{\mu}} \times \mbbS)^* \}$, 

      $\mbbQ = \mbbQ_0 \cup \mbbQ_1$
  \end{center}
  and for all $(u,p,r), (u', p', r') \in \mbbQ$, let $(u,p,r) \leq_{\mbbQ} (u', p', r')$ 
  if $u' \subseteq u$ and 
  $(p, r) \leq_{\mbbP_\mu \times \mbbS} (p', r')$.
  $\leq_{\mbbQ_i}$ is the restriction of $\leq_{\mbbQ}$ on $\mbbQ_i$.

  Note that $(p, r) \mapsto (\dot{\sigma}_{(p,r)}, p,r)$ is an isomorphism from $( {\mbbP_{\mu}} \times \mbbS)^*$ to 
  $\mbbQ_1$.

  By the same proof as Lemma 2.13 in \cite{MR3743612}, we can show that $\mbbQ_0$ is a complete subforcing of $\mbbQ$.
  
  Define $\pi \colon \mbbQ \rightarrow \mbbQ_0$ by 
  $\pi (\dot{\sigma}_{(p, r)}, p, r ) = (\dot{\sigma}_{(p, r)}, \mathbbm{1}_{\mbbP}, {\mathbbm{1}}_{\mbbS} )$.
  Then by the same proof as Lemma 2.12 in \cite{MR3743612}, we can show that $\pi$ is a projection.

  Since $(\mbbP_\mu \times \mbbS)^*$ is dense in ${\mbbP_\mu \times \mbbS}$,  $(\mbbP_\mu \times \mbbS)^*$ is isomorphic to $\mbbQ_1$ and $\mbbQ_1$ is dense in $\mbbQ$,
  we can take a $\mbbQ$-name $\dot{\sigma}_{\mbbQ}$ corresponding to $\dot{\sigma}$.

  Let $\dot{G}_0$ be the canonical $\mbbQ_0$-name for a $\mbbQ_0$-generic filter.

  We can show that $\mbbB(\dot{\sigma}_\mbbQ)$, the complete Boolean subalgebra of $\mbbB(\mbbQ)$ induced by $\dot{\sigma}_\mbbQ$,
  is equal to $\mbbB(\mbbQ_0)$, the complete Boolean subalgebra of $\mbbB(\mbbQ)$ generated by $\mbbQ_0$.
  Hence it suffices to show that 
  $(\mbbQ / \dot{G}_0)_{\id}$ is $\mu$-closed.
  It holds that the quotient $(\mbbQ / \dot{G}_0)_{\id}$ by a complete subforcing is equal to $(\mbbQ / \dot{G}_0)^\pi$, the quotient by a projection $\pi$.

  To see that $(\mbbQ / \dot{G}_0)^\pi$ is $\mu$-closed, we take 
  $G_0$ a $\mbbQ_0$-generic filter over $V$ arbitrary.
  By the definition of $\pi$,
  \begin{center}
    $(\mbbQ / G_0)^\pi =
    \{(\dot{\sigma}_{(p,r)}, p, r ) \in \mbbQ \mid (\dot{\sigma}_{(p,r)}, {\mathbbm{1}_{\mbbP_\mu}}, \mathbbm{1}_\mbbS) \in G_0\}$.
  \end{center}
  It follows that the above set is $\mu$-closed from the $\mu$-closedness of $\mbbQ$.
  \hfill$\square$

  \vspace{0.5\baselineskip}
  
 \noindent\textit{Sketch of Proof of Fact \ref{fact:psp}}.
 We work in $V[G]$. Suppose that $z \in {}^\mu V$ and
\begin{center}
  $A  = (A^\mu_{\varphi, z})^{V[G]} = \{x \in (\mu^\mu)^{V[G]} \mid V[G]\models \varphi(x,z)\}$
\end{center}
has size $\mu^+$.
We show that $A$ contains a perfect subset.

We can take $\nu < \kappa$  with $z \in V[G_{\nu}]$ and an $\add(\mu, 1)\times \coll(\mu, {<}\kappa)$-generic filter $g \times h$ over $V[G_{\nu}]$ such that 
\begin{enumerate}
    \item $V[G]=V[G_{\nu}][g\times h]$,
    \item $A \cap V[G_{\nu}] \subsetneq A \cap V[G_{\nu}][g]$.
\end{enumerate}

In $V[G_{\nu}]$, we take an $\add(\mu, 1)$-name $\dot{\sigma}$ such that
\begin{enumerate}
  \item $\Vdash_{\add(\mu, 1)} ``\dot{\sigma} \notin V[G_{\nu}]$",
  \item $\Vdash_{\add(\mu, 1)}``\Vdash_{\coll(\mu, {<} \kappa)} \varphi(\check{\dot{\sigma}}, z)''$.
\end{enumerate}

Recall that $\mbb{P}_\mu$ is forcing equivalent to $\add ( \mu , 1 )$. Let $k$ be a $\mbb{P}_\mu$-generic filter over $V[G_{\nu} \times h]$ with 
\begin{center}
    $V[G_{\nu}\times k \times h] = 
    V[G_{\nu}\times g \times h]$.
\end{center}

 By Fact \ref{lem:gen2.18}, every branch $x$ of $T_{k}$ in $V[G]$ is $\add(\mu, 1)$-generic over $V[G_{\nu} \times h]$, and so it is generic over $V[ G_\nu ]$.
 We show that $\dot{\sigma}[x] \in A$ for each branch $x\in [T_k]$ in $V[G]$.
 Then
 \begin{center}
 $\{\dot{\sigma}[x] \mid x\in [T_k]^{V[G]}\}$
 \end{center}
 is a perfect subset of $A$.
 By the choice of $\dot{\sigma}$, it suffices to see that $V[G]$ is an extension of $V[G_{\nu}][x]$ by $\coll(\mu, {<}\kappa)$ 
 for each branch $x\in [T_k]$ in $V[G]$.
 
  Suppose that $x$ is a branch of $T_k$ in $V[G]$.
  Take $\theta < \kappa$ such that
      $V[G_{\nu}][x] \subseteq V[G_{\theta+1}]$
  and, in $V$, $\theta$ is a regular cardinal with $\theta ^ {<\mu} = \theta$.
  By Fact \ref{lem:gen2.18},
  $V[G_{\theta+1}]$ is an extension of $V[G_\nu][x]$ by a poset $\mbbP$ 
  which is $\mu$-closed. 
  Since $\mbbP$ collapses cardinals in $[\nu, \theta]$ and has the size $\theta^{<\mu} = \theta$,
  $\mbbP$ is forcing equivalent to $\coll(\mu, {<}\ \theta+1)$.
  Thus $\mbbP \times \coll (\mu, [\theta+1, \kappa))$ is forcing equivalent to $\coll(\mu, {<}\kappa)$,
  so $V[G]$ is an extension of $V[ G_\nu ][x]$ by $\coll(\mu, {<}\kappa)$.
  \hfill$\square$

  \vspace{0.5\baselineskip}

  A point of the above proof is that $V[G]$ is a forcing extension of $V[ G_\nu ][x]$ by $\coll ( \mu , {<}\ \kappa )$ for any branch $x$ of $T_k$ in $V[G]$. The original proof in \cite{MR3743612} only shows that $V[G]$ is a $\mu$-closed forcing extension of $V[ G_\nu \times h ][x]$ for any such $x$. This does not seem to imply that $V[G]$ is an extension of $V[ G_\nu ][x]$ by $\coll ( \mu , {<}\ \kappa )$.

\subsection{Perfect set dichotomy in the generalized Solovay model}\label{sec:psd_genSolovay}
 We show the perfect set dichotomy in the generalized Solovay model.
 \begin{theorem}\label{thm:psd_genSolovay}
  Assume $(\ast)_{\mu, \kappa}$ and $\omega<\mu$.
  In $\VmuVG$ suppose $E$ is an equivalence relation on $\mu^\mu$.
  Then either $\mu^\mu / E$ is well-orderable, or else there is an $E$-inequivalent perfect subset of $\mu^\mu$.
 \end{theorem}

 \begin{proof}
  We work in $V[G]$.
  Let $E$ be an equivalence relation on $\mu^\mu$ in $\VmuVG$.
  Below, $\mcalH_\kappa$ denotes the one in $V$.
  Take a formula $\varphi$, $v \in V$ and $r \in \mu^\mu$ such that $x \, E \, y$ if and only if $\varphi (v,r,x,y)$ for all $x,y \in \mu^\mu$.
  By the $\kappa$-c.c. of $\coll (\mu, {<}\kappa)$
  we can take $\zeta\in [\mu, \kappa)$ such that 
  $r \in V[G_{\zeta}]$.
  Since $V[G]$ is a $\coll(\mu, {<}\kappa)$-generic extension of $V[G_{\zeta}]$,
  by replacing the ground model $V$ with $V[G_{\zeta}]$ we may assume that $r \in V$.
  So we may omit $r$.
  We may also assume $\kappa$ is definable from $v$.

  Let $\psi (v,x,y)$ be the formula stating that $\Vdash_{\coll(\mu, {<}\kappa)} \varphi ( \check{v} , \check{x} , \check{y} )$.
  By Proposition \ref{prop:coll_unctble},
  if $z \in (\mcalH_\kappa)^{V[G]}$ is nice then for all $x, y\in V[z]$
  we have
  \begin{center}
      $x E y 
      \,
      \Leftrightarrow
      \, 
      V[z] \models \psi (v, x, y)$.
  \end{center} 

  We divide the proof into two cases, as in the proof of \ref{dic_equi_Solovay}.
  The difference is that here we only deal with intermediate models that are nice.

  In $V$, let $\Omega$ be the set of all triples $( \xi , p , \dot{x} )$ such that 
  $\xi \in [\alpha, \kappa)$, $p \in \coll (\mu, \xi)$ and $\dot{x}$ is a 
  $\coll (\mu, \xi)$-name for an element of $\mu^\mu$ with $\dot{x} \in \mcalH_\kappa$.
  For $( \xi , p , \dot{x} ) \in \Omega$,
  \[
  \Eval^* (\xi,p,\dot{x}) :=
  \{ \dot{x}[h] \mid \mbox{$h$ is a nice
  $\coll (\mu, \xi)$-generic filter over $V$ with $p\in h$} \} \, .
  \]

  Note that $\Eval^* \in \VmuVG$ since the niceness is definable in $\VmuVG$.

  A non-empty set $X\subseteq\mu^\mu$ is called an $E$-\textit{component} if $x \, E \, y$ for all $x, y \in X$.

  \vskip.5\baselineskip

  \noindent\underline{(Case I)} For any $( \xi , q , \dot{x} ) \in \Omega$ there is $p \leq q$ such that $\Eval^* (\xi,p,\dot{x})$ is an $E$-component.
  
  \vskip.5\baselineskip

  \begin{claim} \label{E_component^*}
  For any $a \in \mu^\mu$ there is $( \xi , p , \dot{x} ) \in \Omega$ such that $a \in \Eval^* ( \xi , p , \dot{x} )$ and $\Eval^* ( \xi , p , \dot{x} )$ is an $E$-component.
  \end{claim}

  \noindent\textit{Proof of Claim}.
  Suppose $a \in \mu^\mu$. Then there is $\xi < \kappa$, a $\coll(\mu, \xi)$-name $\dot{x}$ and a
  nice $\coll (\mu, \xi)$-generic filter $h$ over $V$ such that $a = \dot{x}[h]$. Let
  \[
  D :=
  \{ p \in \coll ( \mu , \xi ) \mid \mbox{$\Eval^* ( \xi , p , \dot{x} )$ is an $E$-component} \} \, .
  \]
  Then $D$ is dense in $\coll ( \mu , \xi )$ by the assumption of Case I, and $D \in V$ as in the proof of Claim \ref{E_component}.
  Take $p \in h \cap D$. Then $( \xi , p , \dot{x} )$ is as desired.
  \hfill $\square$(Claim \ref{E_component^*})

  \vskip.5\baselineskip

  Since $\AC$ holds in $V$, we can take a well-ordering $\leq_\Omega$ of $\Omega$.
  By Claim \ref{E_component}, for each $A \in \mu^\mu / E$, there is $( \xi , p , \dot{x} ) \in \Omega$ such that $\Eval^* ( \xi , p , \dot{x} ) \subseteq A$. 
  For $A \in \mu^\mu / E$, let $( \xi_A , p_A , \dot{x}_A )$ be the $\leq_\Omega$-least such $( \xi , p , \dot{x} ) \in \Omega$.
  For $A,B \in \mu^\mu / E$, let $A \trianglelefteq B$ if $( \xi_A , p_A , \dot{x}_A ) \leq_\Omega ( \xi_B , p_B , \dot{x}_B )$.
  Then $\trianglelefteq$ well-orders $\mu^\mu / E$.
  Since $\leq_\Omega$ and $ \Eval^*$ belong to $ \VmuVG$, $\trianglelefteq$  belongs to $\VmuVG$.

  \vskip.5\baselineskip
  
  \noindent\underline{(Case II)} There is $( \xi , q , \dot{x} ) \in \Omega$ such that $\Eval^* ( \xi , p , \dot{x} )$ is not an $E$-component for any $p \leq q$.

  \vskip.5\baselineskip
 
  Take $(\xi, q, \dot{x}) \in \Omega$ as in the assumption.
  Let $\dot{x}_\mathrm{left}$ and $\dot{x}_\mathrm{right}$ be $\coll(\mu, \xi)\times \coll(\mu, \xi)$-names in $V$ as in the proof of Theorem \ref{dic_equi_Solovay}.

  \vskip.5\baselineskip

  \begin{claim} \label{claim:nice_mutually}
  Suppose $h_0, h_1$ are  
  nice $\coll(\mu, \xi)$-generic filters over $V$ and $p \in \coll(\mu, \xi)$.
  Then there is a $\coll(\mu, \xi)$-generic filter
   $k$ over $V$ with $p \in k$ such that $h_1 \times k$ and $h_2 \times k$ are nice $\coll(\mu, \xi) \times \coll(\mu, \xi)$-generic filters over $V$.
  \end{claim}

    \noindent\textit{Proof of Claim}.
   Take $\nu < \kappa$ such that $h_0, h_1 \in V[G_{\nu + 1}]$ and, in both $V[ h_0]$ and $V[ h_1 ]$, $\nu$ is a regular cardinal with $\nu^{<\mu}=\nu$.
   By Lemma \ref{lem:nu+1}, we can take $\coll(\mu, {<}\ \nu+1)$-generic filters $l_0$ and $l_1$ over $V[h_0]$ and $V[h_1]$ respectively such that
   \begin{center}
    $V[G_{\nu+1}] = V[h_0][l_0] = V[h_1][l_1]$.
   \end{center}
   Let $H$ be a $\coll(\mu, {<}\kappa)$-generic filter over $V[G_{\nu+1}]$ such that $V[G]=V[G_{\nu+1}][H]$ and let $k' = H \cap \coll(\mu, \xi)$.
   Then $k'$ is $\coll(\mu, \xi)$-generic over $V[G_{\nu+1}]$, thus generic over $V[h_0]$ and $V[h_1]$.
   
   We see that $h_i \times k'$ is nice for each $i=0,1$.
   Suppose $y\in \mcalH_\kappa^{V[G]}$.
   We can take a cardinal  $\alpha$ with $\nu < \alpha < \kappa$ with $y \in V[G_{\nu+1}][H_\alpha]$. Then
   \begin{center}
    $V[G_{\nu+1}][H_\alpha]=V[h_i][l_i][k'][H_{(\xi,\alpha)}] = V[h_i][k'][l_i][H_{(\xi, \alpha)}]$
   \end{center}
   for $i=0,1$.
   Thus $y$ is contained by a $\mu$-closed forcing extension of $V[h_i \times k']$.
   
   By the homogeneity of $\coll (\mu, \xi)$, 
   we can take a $\coll (\mu, \xi)$-generic filter $k$ over $V$ such that $p \in k$ and $V[k] = V[k']$, which is as required.
  \hfill $\square$(Claim \ref{claim:nice_mutually})
 
  \vskip.5\baselineskip

  \begin{claim} \label{claim:nice_product_force}
  In $V$, $(q,q) \Vdash_{\coll(\mu,\xi) \times \coll(\mu, \xi)} \lnot \psi ( \check{v} , \dot{x}_\mathrm{left} , \dot{x}_\mathrm{right} )$.
  \end{claim}

  \noindent\textit{Proof of Claim}.
  This proof is the same as the proof of Claim \ref{claim:product_force}.
  
  Assume the claim fails. Then there is $(q_0,q_1)\leq (q,q)$ such that
  
  \begin{center}
    $(q_0,q_1) 
    \Vdash_{\coll(\mu, \xi)\times \coll(\mu, \xi)} \psi ( \check{v} , \dot{x}_\mathrm{left} , \dot{x}_\mathrm{right} )$
    \end{center}
  
  \noindent
  in $V$. By the choice of $( \xi , q , \dot{x} )$ and $q_0 \leq q$, $\Eval^* ( \xi, q_0, \dot{x} )$ is not an $E$-component.
  Hence there are nice $\coll(\mu, \xi)$-generic filters $h_0, h_1$ over $V$ such that $q_0 \in h_0, h_1$ and $\lnot ( \dot{x} [ h_0 ] \, E \, \dot{x} [ h_1 ] )$.  
  By Claim \ref{claim:nice_mutually}, we can take a nice $\coll(\mu, \xi)$-generic filter $k$ over $V$ with $q_1 \in k$ such that 
  $h_0 \times k$ and $h_1 \times k$ are nice $\coll(\mu, \xi) \times \coll(\mu, \xi)$-generic filters over $V$.

  Since $( q_0 , q_1 ) \in h_i \times k$,
  \begin{center}
    $V[h_i \times k] \models \,  
    \psi (\check{v}, \dot{x}_\mathrm{left}, \dot{x}_\mathrm{right})$
  \end{center}for $i = 0,1$.
  By the niceness of $h_i \times k$ and Proposition \ref{prop:coll_unctble},
  $\dot{x}[h_i] \, E \, \dot{x}[k]$ holds in $V[G]$.
  Then $\dot{x} [ h_0 ] \, E \, \dot{x} [k] \, E \, \dot{x} [ h_1 ]$.
  This contradicts the choice of $h_0$ and $h_1$.
  \hfill $\square$(Claim \ref{claim:nice_product_force})

  \vskip.5\baselineskip

  We construct a family of $\coll(\mu, \xi)$-generic filters and obtain a desired perfect set as interpretations of $\dot{x}$ by them.

  Let $\alpha < \kappa$ be a successor cardinal with $\xi <\alpha$.
  Note that $\add(\mu, 1)$, $\coll(\mu, \xi)$ and $\mbbP_\mu$ are forcing equivalent in $V[G_\alpha]$.
  Since $\coll(\mu, [\alpha, \kappa))$ and $\mbbP_\mu \times \coll(\mu, [\alpha, \kappa))$ are forcing equivalent in $V[G_\alpha]$, we can take a $\mbbP_\mu \times \coll(\mu, [\alpha, \kappa))$-generic filter $H\times G'_{[\alpha, \kappa)}$ over $V[G_\alpha]$ such that 
  \begin{center}
  $V[G] = V[G_\alpha][H][G'_{[\alpha, \kappa)}]$.
  \end{center}

  Applying Fact \ref{lem:gen2.18} as $\mbbS = \coll(\mu, [\alpha, \kappa))$,
  for each distinct $b, c \in [T_H]^{V[G]}$,
  $\langle b, c\rangle$ is nice and $\add(\mu, 1) \times \add(\mu, 1)$-generic over $V[G_\alpha]$.
  
  Since $\add(\mu, 1)$ is forcing equivalent to $\coll(\mu, \xi)$ in $V[G_\alpha]$,
  for $b \in [T_H]$
  we can take a nice $\coll(\mu, \xi)$-generic filter $h_b$ over $V[G_\alpha]$ corresponding to $b$.
  By the homogeneity of $\coll (\mu, \xi)$, we can assume that $q \in h_b$ for all $b \in [T_K]$.

  Then by Claim \ref{claim:nice_product_force}, 
  $X := \{ \dot{x}[h_b] \mid b\in [T_K] \}$ is an $E$-inequivalent subset of $\mu^\mu$.
  Thus $X$ is an injective image of a perfect set $[T_K]$.
  By Fact \ref{fact:psp}, $X$ contains a perfect subset.
  \end{proof}

 \section{Linear orders in Solovay model}

 Woodin showed the following dichotomy theorem under the assumption of $\ZF+\AD+V=\LR$.
 The argument can be founded in \cite{chan2021cardinality}.

 \begin{fact}[Woodin]\label{fact:dic_set_AD}
  Assume $\ZF+\AD+V=\LR$.
  For any set $X$, either $X$ is well-orderable, or else there is an injection from $\mbbR$ to $X$.
 \end{fact}

 In this section, we prove that this dichotomy for sets holds in the generalized Solovay model for a weakly compact cardinal.
 
\begin{theorem}\label{thm:dic_set_Solovay}
  Assume $(\ast)_{\mu,\kappa}$.
  In $\VmuVG$, for any set $X$, either $X$ is well-orderable, or else there is an injection from $\mu^\mu$ to $X$.  
\end{theorem}

\begin{proof}
  We work in $\VmuVG$. Take an arbitrary set $X$.
  By Fact \ref{fact:surj} we can take an ordinal $\lambda$ and a surjection $f \colon V_\lambda \times \mu^\mu\rightarrow X$, where $V_\lambda$ is the one in $V$.
  For $v\in V_\lambda$ let $X_v = f[\{v\}\times \mu^\mu]$.
  
  If there is $v\in V_\lambda$ such that $\mu^\mu$ injects into $X_v$, then $\mu^\mu$ injects into $X$ clearly.
  Thus we suppose that for any $v\in V_\lambda$ there is no injection from $\mu^\mu$ to $X_v$.
  
  For each $v \in V_\lambda$ define an equivalence relation $E_v$ on $\mu^\mu$ by $x \, E_v \, y$ if $f(v,x) = f(v,y)$.
  Note that $[x]_{E_v} \mapsto f(v,x)$ is a bijection from $\mu^\mu / E_v$ to $X_v$.
  So there is no injection from $\mu^\mu$ to $\mu^\mu / E_v$.
  By Theorem \ref{dic_equi_Solovay} $(\mu=\omega)$ and \ref{thm:psd_genSolovay} $(\mu>\omega)$, $\mu^\mu / E_v$ is well-orderable.
  Moreover, from the proof of Theorem \ref{dic_equi_Solovay} and \ref{thm:psd_genSolovay}, we can construct a well-ordering on $\mu^\mu /E_v$ uniformly for $v\in V_\lambda$.
  Then we have a sequence $\langle \leq_v \mid v \in V_\lambda \rangle$ such that each $\leq_v$ is a well-ordering of $X_v$.
  Note also that there is a well-ordering $\leq_\lambda$ of $V_\lambda$ since $\AC$ holds in $V$, and $V \subseteq \VmuVG$.
  Thus we can easily construct a well-ordering of $X = \bigcup_{v \in V_\lambda} X_v$ from $\leq_\lambda$ and $\langle \leq_v \mid v \in V_\lambda \rangle$.
\end{proof}

Applying this dichotomy theorem, we obtain the basis theorem for ``uncountable" linear orders in the generalized Solovay model.
We see the following partition properties of $\kappa$ and $\mbbR$.
For a function $g \colon [A]^2 \rightarrow 2$, a subset $B$ of $A$ is said to be $g$-\textit{homogeneous} if $g \upharpoonright [B]^2$ is constant.
\begin{lemma}\label{lem:partition_kappa}
  Assume $(\ast\ast)_{\mu,\kappa}$. In $\VmuVG$, $\mu^+ \rightarrow \left( \mu^+ \right)^2_2$. 
\end{lemma}

\begin{proof}
    We work in $V[G]$.
    Note that $\kappa=\mu^+$ in $\VmuVG$.
    Fix $f \colon [\kappa]^2 \rightarrow 2$ in $\VmuVG$.
    We can take $r \in \mu^\mu$ such that $f \in \mrm{HOD}^{V[G]}_{V\cup \{r\}}$.
    Let $\xi$ be an ordinal with $\mu<\xi<\kappa$ and $r\in V[G_\xi]$.
    Using homogeneity,
    we can take a formula $\psi$ and $v \in V$ such that for any $\{ \alpha , \beta \} \in [ \kappa ]^2$ and each $i = 0,1$, we have $f( \{ \alpha , \beta \} ) = i$ if and only if $\psi ( v , r , \alpha , \beta , i )$ holds in $V[G_\xi]$.
    So $f \in V[G_\xi]$.
    
    Since $\kappa$ remains weakly compact in $V[G_\xi]$, 
    in $V[G_\xi]$, there is an $f$-homogeneous $H \subseteq \kappa$ of order-type $\kappa$.
    Since $V[G_\xi] \subseteq \VmuVG$, $H$ is as desired.
\end{proof}

We let $\langle X\rangle^\gamma$ denote the set of all sequences $\langle x_i \mid i<\gamma \rangle$ of distinct elements of $A$ for any set $X$ and ordinal $\gamma$. 

\begin{fact}[Schlicht \cite{MR3743612}]\label{fact:perfect_conti}
  Suppose that $\mu$ is an uncountable regular cardinal, $\mbbQ$ is a ${<}\mu$-distributive forcing and $\gamma < \mu$.
  Let $H$ be an $\add(\mu, 1)\times \mbbQ$-generic filter over $V$.
  Then in $V[H]$, for every function $f \colon \langle \mu^\mu \rangle^\gamma \mapsto \mu^\mu$ that is definable from an element of $V$,
  there is a perfect set $C \subseteq \mu^\mu$ such that $f\restrict \langle C \rangle^\gamma$ is continuous.
\end{fact}

\begin{lemma}\label{lem:partition_real}
    Assume $(\ast)_{\mu,\kappa}$. 
    In $\VmuVG$ let $[\mu^\mu]^2 = P_0 \sqcup P_1$ be a partition.
    Then there is a perfect set $C \subseteq \mu^\mu$ such that $[C]^2 \subseteq P_i$ for some $i = 0,1$. 
 \end{lemma}

 \begin{proof}
  In case of $\mu=\omega$, the statement follows from the following fact and the Baire property in the Solovay model. 
  The proof can be found in Theorem 19.7 in \cite{kechris2012classical}.
  \begin{fact}[Galvin]\label{fact:Galvin}
    Let $[\omega^\omega]^2 = P_0 \sqcup P_1$ be a partition.
    If $P_0$ and $P_1$ have the Baire property,
    then there is a perfect set $C \subseteq \mbbR$ such that $[C]^2 \subseteq P_i$ for some $i = 0,1$.
  \end{fact}
  Now we assume that $\mu > \omega$.
  We may assume that we can define $P_0$ and $P_1$ in $V[G]$ using parameters in $V$.
  Suppose $[\mu^\mu]^2 = P_0 \sqcup P_1$ is a partition in $\VmuVG$.
  Define $f \colon \langle\mu^\mu\rangle^2 \rightarrow 2$ by $f(\langle x,y \rangle)=i$ if and only if $\{x,y\} \in P_i$ for $i=0, 1$.
  Note that $f(\langle x, y \rangle)=f(\langle y, x \rangle)$ for all $x \neq y \in \mu^\mu$.
  Take $r\in \mu^\mu$ such that $f \in \mathrm{HOD}^{V[G]}_{V\cup \{r\}}$ and let $\xi$ be an ordinal with $\mu<\xi<\kappa$ such that $r\in V[G_\xi]$.
  Since $V[G]$ is a $\coll(\mu, {<}\kappa)$-generic extension of $V[G_\xi]$ and $\coll(\mu, {<}\kappa)$ is forcing equivalent to $\add(\mu, 1) \times \coll(\mu, {<}\kappa)$,
  by Fact \ref{fact:perfect_conti}, we can take a perfect set $C \subseteq \mu^\mu$ in $V[G]$ such that $f\restrict \langle C \rangle^2$ is continuous.

  Let $Q_i = P_i \cap [C]^2$.
  Then $Q_i$ is the inverse image of $\{i\}$ by $f$ in $[C]^2$, so $Q_i$ is open in $[C]^2$.
  By replacing $\mu^\mu$ by $C$ if necessary, we may assume that $P_i$ is open.

  If there is a non-empty open set $U \subseteq \mu^\mu$
  with $[U]^2 \subseteq P_0$, then any perfect subset of $U$ is as desired.
  So we assume that $[U]^2 \cap P_1 \neq \emptyset$ for all non-empty open sets $U \subseteq \mu^\mu$.
  Since $P_1$ is open, we can take $s,t \in \mu^{<\mu}$ such that $N_s, N_t \subseteq U$ and $N_s \times N_t \subseteq P_1$ for each non-empty open sets $U$.
  
  Repeating this, we can construct 
  $\langle u_s \in \mu^{<\mu} \mid s\in 2^{<\mu}\rangle$ such that 
  \begin{itemize}
    \item $|u_s| > |s|$,
    \item $N_{u_{s \,  \hat{} \,  0}} \times N_{u_{s \, \hat{} \, 1}} \subseteq P_1$ for $s \in 2^{<\mu}$,
    \item $u_s = \bigcup_{\alpha < |u|} u_{s\restrict\alpha}$ for $s\in 2^{<\mu}$.
  \end{itemize}
  For $x \in 2^\mu$ let $u_x = \bigcup_{\alpha<\mu} u_{x\restrict \alpha}$.
  By construction, $C = \{u_x \in \mu^\mu \mid x\in 2^\mu\}$ is perfect and $[C]^2 \subseteq P_1$.
 \end{proof}

\begin{theorem}\label{Solovay_basis}
  Assume $(\ast\ast)_{\mu,\kappa}$.
    In $\VmuVG$, for any linearly ordered set $X$ such that there is no injection from $X$ to $\mu$, at least one of $( \mu^+, \leq )$, $( \mu^+, \geq )$ or $( 2^\mu , \leq_{\mathrm{lex}} )$ is embeddable into $X$, where $\leq_{\mathrm{lex}}$ is the lexicographic ordering.
\end{theorem}

\begin{proof}
  Take an arbitrary linearly ordered set $X$.
  By Theorem \ref{thm:dic_set_Solovay}, we can divide the proof into two cases.

  \vskip.5\baselineskip
  
  \noindent\underline{(Case I)} $X$ is well-orderable.

  \vskip.5\baselineskip

  We can take an injection $f\colon \mu^+ \rightarrow X$.
  Define $g\colon [\mu^+]^2 \rightarrow 2$ by
    \[
    g(\{\alpha,\beta\})=\left\{
    \begin{array}{cc}
      0   & (f(\alpha) <_X f(\beta))\\
      1   & (f(\alpha) >_X f(\beta))
    \end{array}
    \right .
    \]
    where $\alpha < \beta$.
    By Lemma \ref{lem:partition_kappa}, there is a $g$-homogeneous $C\subseteq \mu^+$ such that $\mathrm{otp}(C)=\mu^+$.
    If $g( \{ \alpha , \beta \} ) = 0$ for all $\{ \alpha , \beta \} \in [C]^2$ ($\alpha < \beta$), then $f \upharpoonright C$ is an embedding from $(C, \leq )$ to $X$. Otherwise, $f \upharpoonright C$ is an embedding from $( C , \geq )$ to $X$.
    Note that $( \mu^+ , \leq )$ is isomorphic to $(C, \leq )$.
    So either $( \mu^+ , \leq )$ or $( \mu^+ , \geq )$ is embeddable into $X$.
    
  \vskip.5\baselineskip

  \noindent\underline{(Case II)} There is an injection from $\mu^\mu$ to $X$.

  \vskip.5\baselineskip

  Take an injection $f \colon \mu^\mu \rightarrow X$.
  By the same argument as (Case I), using Fact \ref{lem:partition_real} instead of Lemma \ref{lem:partition_kappa}, there is a perfect $C \subseteq \mu^\mu$ such that either $(C, \leq_\mathrm{lex} )$ or $(C, \geq_\mathrm{lex} )$ is embeddable into $X$. 
  Since $( 2^\mu , \leq_\mathrm{lex} )$ is embeddable into both $(C, \leq_\mathrm{lex} )$ and $(C, \geq_\mathrm{lex})$,
  $( 2^\mu , \leq_\mathrm{lex})$ is embeddable into $X$.
\end{proof}

 $(\mbbR, <_{\mbbR})$ is embeddable into $(2^\omega, \leq_{\mathrm{lex}})$ via the binary expansion. 
 Hence, we obtain the following corollary in the Solovay model.

\begin{cor}\label{cor:three_element_basis_Sol}
  Assume $(\ast\ast)_{\omega,\kappa}$.
    In $\VomegaVG$, for any uncountable linear ordered set $X$, at least one of $( \omega_1, \leq )$, $( \omega_1, \geq )$ or $( \mbbR , \leq_{\mbbR} )$ is embeddable into $X$.
\end{cor}

 Note that the same argument works under the assumption of $\ZF+\AD + V=\LR$, 
 using Fact \ref{fact:dic_set_AD},
 \ref{fact:Galvin} and the well-known fact that $\omega_1 \rightarrow (\omega_1)^2_2$ holds under $\ZF+\AD$. Weinert also noticed this independently.

 \begin{theorem}\label{thm:three_element_basis_AD}
  Assume $\ZF+ \AD + V= \LR$.
  For any uncountable linear ordered set $X$, at least one of $( \omega_1, \leq )$, $( \omega_1, \geq )$ or $( \mbbR , \leq_\mbbR )$ is embeddable into $X$.
\end{theorem}

 An uncountable linear order is an \textit{Aronszajn line} if it does not contain any suborder isomorphic to $(\omega_1, \leq)$, $(\omega_1, \geq)$ or an uncountable subset of $(\mbbR, \leq_\mbbR)$.

 By Corollary \ref{cor:three_element_basis_Sol} and Theorem \ref{thm:three_element_basis_AD}, we obtain the following.

 \begin{cor}
  Assume $(\ast\ast)_{\omega,\kappa}$.
  In $\VomegaVG$, there are no Aronszajn lines.
 \end{cor}

 \begin{cor} Assume $\ZF+ \AD + V= \LR$. There are no Aronszajn lines.
 \end{cor}

\bibliographystyle{plain}
\bibliography{dichotomy}
\end{document}